\documentclass[11pt]{amsart}
\title[Extended  simplicial rational Nomizu's Theorem]
{Extended  simplicial rational Nomizu's Theorem}

\author{Hisashi Kasuya}

\usepackage{amssymb}
\usepackage{amsmath}
\usepackage{amscd}
\usepackage{amstext}
\usepackage{amsfonts}
\usepackage[all]{xy}

\theoremstyle{plain}

\theoremstyle{plain}

\theoremstyle{plain}

\theoremstyle{plain}
\newtheorem{theorem}{Theorem}[section] 
\theoremstyle{remark}
\newtheorem{remark}{Remark}[section] 
\theoremstyle{remark}
\newtheorem{Important note}[theorem]{Important note}
\theoremstyle{Main result}
\newtheorem{main result}{Main result}
\theoremstyle{lemma}

\theoremstyle{definition}

\theoremstyle{proposition}
\newtheorem{proposition}[theorem]{Proposition}
\theoremstyle{corollary}
\newtheorem{corollary}[theorem]{Corollary}
\theoremstyle{remark}
\newtheorem{example}{Example}
\theoremstyle{plain}
\newtheorem{newexample}{New Examples}

\address[Hisashi Kasuya]{Department of Mathematics, Tokyo Institute of Technology, 1-12-1, O-okayama, Meguro, Tokyo 152-8551, Japan}
\email{kasuya@math.titech.ac.jp}


\newcommand{\C}{\mathbb{C}}
\newcommand{\R}{\mathbb{R}}

\newcommand{\Q}{\mathbb{Q}}
\newcommand{\Z}{\mathbb{Z}}

\newcommand{\n}{\frak{n}}

\begin{document} 

\maketitle
\begin{abstract}

For a torsion-free virtually polycyclic group $\Gamma$,
we give a canonical homomorphism form  certain finite-dimensional cochain complex to the $\Q$-polynomial de Rham complex of the simplicial classifying space $B\Gamma$ which induces a cohomology isomorphism.
By this result, we obtain the Sullivan's  minimal model of certain differential graded algebra  defined on $B\Gamma$ and 
we obtain new examples of hard Lefschetz  symplectic manifolds and hard Lefschetz  contact manifolds.
\end{abstract}

\section{Introduction}

A group $\Gamma$ is polycyclic if it admits a sequence 
\[\Gamma=\Gamma_{0}\supset \Gamma_{1}\supset \cdot \cdot \cdot \supset \Gamma_{k}=\{ e \}\]
of subgroups such that each $\Gamma_{i}$ is normal in $\Gamma_{i-1}$ and $\Gamma_{i-1}/\Gamma_{i}$ is cyclic.
For a polycyclic group $\Gamma$, we denote ${\rm rank}\,\Gamma=\sum_{i=1}^{i=k} {\rm rank}\,\Gamma_{i-1}/\Gamma_{i}$.
Let $\Gamma$ be a torsion-free virtually polycyclic group.
For a representation of $\Gamma$ into a $\Q$-algebraic group $G$,
 if the image $\rho(\Gamma)$ is Zariski-dense in $G$, then  we have $\dim U\le {\rm rank}\,\Gamma$ where $U$ is the  unipotent radical of $G$  (see \cite[Lemma 4.36.]{R}).
We say that  $\rho:\Gamma\to  G$ is a full representation if $\dim U= {\rm rank}\,\Gamma$.

For a  a simplicial complex $K$ with $\pi_{1}K=\Gamma$ and a $\Gamma$-module $V$ with a finite dimensional $\Q$-vector space, considering $V$ as a local system on $K$,
we can define  the $\Q$-polynomial de Rham complex  $A^{\ast}_{p}(K, V)$ of $K$ with values in local system $V$.
We also consider $V=\varinjlim V_{i}$ for an inductive system of finite dimensional $\pi_{1}(K)$-modules $V_{i}$.
We define
\[A^{\ast}_{p}(K, V)=\varinjlim A^{\ast}_{p}(K, V_{i}).
\]

Let $\Gamma$ be a torsion-free virtually polycyclic group.
We consider the following situation:
\begin{itemize}
\item We have a $\Q$-algebraic group $G$ and an injective representation $\rho :\Gamma\to G$ such that the image $\rho(\Gamma)$ is Zariski-dense in $G$.
\item $\rho:\Gamma\to  G$ is a full representation.
\end{itemize}
We consider the simplicial classifying space $B\Gamma$ of a torsion-free virtually polycyclic group $\Gamma$.
For a rational  $G$-module $V$,
we consider the $\Q$-polynomial de Rham complex  $A^{\ast}_{p}(B\Gamma, V)$.
We also consider the complex of "$G$-invariant differential forms" on $U$.
We can take
a splitting $G=T\ltimes U$ such that $U$ is the unipotent radical of $G$ and $T$ is a maximal reductive subgroup of $G$ (see \cite{Mos1}).
For the Lie algebra $\frak u$ of $U$, we consider the cochain complex
$\left(\bigwedge {\frak u}^{\ast}\otimes V\right)^{T}$  of the $T$-invariant elements of the cochain complex of the Lie algebra ${\frak u}$ with values in $V$.
Then, in this paper we show the following result.
\begin{theorem}\label{MINt}
We have an explicit  map $\left(\bigwedge {\frak u}^{\ast}\otimes V\right)^{T}\to  A^{\ast}_{p}(B\Gamma,V)$ which induces a cohomology isomorphism.
\end{theorem}

\begin{remark}
It is known that a $\Q$-algebraic group $G$ and an injective representation $\rho :\Gamma\to G$  so that
\begin{itemize}
\item We have a $\Q$-algebraic group $G$ and an injective representation $\rho :\Gamma\to G$ such that the image $\rho(\Gamma)$ is Zariski-dense in $G$.
\item $\rho:\Gamma\to  G$ is a full representation.
\item  The centralizer $Z_{G}(U)$ of $U$ is contained in $U$.
\end{itemize}
 exist and such $G$ is unique up to isomorphism of $\Q$-algebraic groups (\cite[Appendix A.]{B}).
Such $G$ is called the algebraic hull of $\Gamma$.

For a representation $\phi:\Gamma\to GL(V)$ with a finite dimensional $\Q$-vector space $V$,
consider the representation $\rho\times \phi:\Gamma\to G\times GL(V)$ and the take the Zariski-closure $G^{\phi}$ of the its image.
Then the representation $\Gamma\to G^{\phi}$ is also a full representation and $\phi:\Gamma\to GL(V)$ is extended to  a rational representation $G^{\phi}\to GL(V)$ (see \cite[Section 3]{KCT}).
Thus, for any representation $\phi:\Gamma\to GL(V)$ with a finite dimensional $\Q$-vector space $V$, applying Theorem \ref{MINt}, the cohomology of the $\Q$-polynomial de Rham complex  $A^{\ast}_{p}(B\Gamma, V)$ is computed by the finite-dimensional cochain complex $\left(\bigwedge {\frak u}^{\ast}\otimes V\right)^{T}$.
\end{remark}

Theorem \ref{MINt} can be regarded as a generalization of  simplicial rational version of Nomizu's Theorem (\cite{Nom}, \cite{LP}).
Let $N$ be a simply connected nilpotent Lie group and $\frak n$ be the Lie algebra of $N$.
We suppose that  $N$ has a lattice (i.e. cocompact discrete subgroup) $\Gamma$.
We consider the nilmanifold $\Gamma\backslash N$.
Then, considering the cochian complex $\bigwedge \frak n^{\ast}$ of the Lie algebra as the space of the invariant differential forms, in \cite{Nom}, Nomizu proves that
the canonical inclusion $\bigwedge \frak n^{\ast}\subset A^{\ast}(\Gamma\backslash N)$ induces a cohomology isomorphism
\[H^{\ast}(\n,\R)\cong H^{\ast}(\Gamma\backslash N,\R)
\]
where $A^{\ast}(\Gamma\backslash N)$ is the de Rham complex of $\Gamma\backslash N$.

In \cite{LP}, Lambe and Priddy give a  simplicial rational version of  Nomizu's theorem.
For simply connected nilpotent Lie group $N$ with a lattice $\Gamma$, $\Gamma$ is a torsion-free finitely generated nilpotent group and a nilmanifold $\Gamma\backslash N$ is an aspherical manifold with the fundamental group $\Gamma$.
Conversely any torsion-free finitely generated nilpotent group $\Gamma$ can be embedded in a simply connected nilpotent Lie group $N$ whose Lie algebra $\frak n$ admits a $\Q$-structure $\frak n_{\Q}$ (see \cite{R}).
In \cite{LP}, considering the simplicial classifying space $B\Gamma$ of a torsion-free finitely generated nilpotent group $\Gamma$ and the $\Q$-polynomial de Rham complex  $A^{\ast}_{p}(B\Gamma,\Q)$,  Lambe and Priddy construct an explicit map $\bigwedge \frak n_{\Q}\to A^{\ast}_{p}(B\Gamma,\Q)$ which induces a cohomology isomorphism.
A simply connected nilpotent Lie group $N$ can be considered as a real unipotent algebraic group with a $\Q$-structure  $N(\Q)$ so that $\Gamma\subset N(\Q)$.
Regarding a torsion-free finitely generated nilpotent group $\Gamma$ as a polycyclic group, we can say that the $\Q$-algebraic group $N(\Q)$ is the algebraic hull of $\Gamma$ (see \cite{R}).

Nomizu's theorem gives an important fact on the theory of Sullivan's minimal model.
We can say that the Differential graded algebra (shortly DGA) $\bigwedge \frak n^{\ast}$ is the  minimal model of  $A^{\ast}(\Gamma\backslash N)$ (see \cite{H}).

We can generalize this fact.
For a  torsion-free virtually polycyclic group $\Gamma$,
take $G$ the algebraic hull of $\Gamma$ and a splitting $G=T\ltimes U$ for a maximal reductive subgroup $T$.
Denote by $\frak u$  the Lie algebra of  the unipotent hull $U$ of $\Gamma$.
Consider the the coordinate ring $\Q[T]$ of the algebraic group $T$.
By the composition $\Gamma\to G\to T$ where $G\to T$ is the projection,
since $\Q[T]$ is a rational $T$-module, we regard $\Q[T]$ as a $\Gamma$-module.
Consider the $\Q$-polynomial de Rham complex  $A^{\ast}_{p}(B\Gamma,\Q[T])$.
Then $A^{\ast}_{p}(B\Gamma,\Q[T])$ is a differential graded algebra (DGA).
By Theorem \ref{MINt}, we have the following result.

\begin{theorem}\label{polmi}
We have an explicit DGA map $\bigwedge {\frak u}^{\ast}\to  A^{\ast}_{p}(B\Gamma,\Q[T])$ which induces a cohomology isomorphism.
Hence $\bigwedge{\frak u}^{\ast}$ is the minimal model of $A^{\ast}_{p}(B\Gamma,\Q[T])$.
\end{theorem}

See \cite{Kas} for the similar result on solvmanifolds where a solvmanifold is the compact  quotient $\Gamma\backslash S$
of a simply connected solvable Lie group $S$ by a lattice $\Gamma$.

An infra-solvmanifold is a manifold of the form $G/\Delta$, where $G$ is a simply connected solvable Lie group, and $\Delta$ is a torsion-free subgroup of ${\rm Aut}(G)\ltimes G$ such that the closure of $h(\Delta)$ in ${\rm Aut}(G)$ is compact where $h:{\rm Aut}(G)\ltimes G\to {\rm Aut}(G)$ is the projection.
An infra-solvmanifold is a  generalization of a solvmanifold.
An infra-solvmanifold is a aspherical manifold with a virtually polycyclic group.
Theorem \ref{MINt} is useful for finding symplectic structures on infra-solvmanifolds $M$ whose cohomology classes in $H^{2}(M,\Z)$.
By finding such symplectic infra-solvmanifolds,
we obtain the following results.
\begin{theorem}
We obtain the following new examples.
\begin{itemize}
\item Symplectic blow-ups of   complex projective spaces along 
certain embedded infra-solvmanifolds which satisfy the hard Lefschetz properties.
(But, it is not clear whether these manifolds admit K\"ahler metrics.)

\item Non Sasakian contact manifolds which satisfy the hard Lefschetz properties in the sense of \cite{CNY}, \cite{Li}.

\end{itemize}

\end{theorem}

\section{Cohomology of algebraic groups}\label{SAL}
The purpose of this section is to construct an explicit cochain complex homomorphism which induces the  Hochschild isomorphism on the rational cohomology of an algebraic group.
For this construction, we are inspired by the simplicial construction of Van Est isomorphism on the continuous cohomology of a Lie group (see \cite{Shul}). 
Let $G$ be a $\Q$-algebraic group.
A  $G$-module is called rational if it is  the sum of finite-dimensional $G$-stable subspaces $\{V_{i}\}$ so that each $V_{i}$ comes from a rational representation. 
For a rational  $ G$-module $V$, we define the rational cohomology $H^{\ast}( G,V)={\rm Ext}^{\ast}_{ G}(\Q,V)$ as \cite{Hoc} and \cite{Jan}.
We have the standard resolution. 
We denote by $C^{p}( G,V)$ the set of the $V$-valued rational functions on $\underbrace{G\times \dots\times G }_{p+1}$ with the left-$ G$-action.
Consider the sequence
\[\xymatrix{
V\ar[r]& C^{0}(G,V)\ar^{d}[r]&C^{1}( G,V)\ar^{d}[r]&\dots
 }
\] 
such that the first map $V\to C^{0}( G,V)$ is the embedding as  constant functions and $d:C^{p}( G,V)\to C^{p+1}( G,V)$ is given by
\[d\phi (g_{0},\dots,g_{p+1})=\sum (-1)^{i}\phi(g_{0},\dots, {\hat g}_{i},\dots, g_{p+1})
\]
for $\phi\in C^{p}( G,V)$, $g_{0},\dots, g_{p+1}\in G$.
Then the rational cohomology $H^{\ast}(G,V)$ is the cohomology of the cochain complex $C^{\ast}( G,V)^{ G}$.
For the unipotent radical $U$ of $G$ and a maximal reductive subgroup $T$, we have a splitting $G=T\ltimes U$ (\cite{Mos1}).
Let $\frak u$ be the Lie algebra of $U$.
Hochschild showed that we have an isomorphism
\[H^{\ast}({\frak u}, V)^{T}\cong H^{\ast}(G,V).
\]
The purpose of this section is to represent this isomorphism 
 as a cochain complex homomorphism.

Let $A^{\ast}_{a}(U)$ be the algebraic de Rham complex of the algebraic variety $U$.
Consider the coordinate ring $\Q[U]$ of $U$.
By the spectral sequence as \cite[Proposition 3.4]{Haia}, we have $H^{1}(G,\Q[U] \otimes W)=H^{1}(U,\Q[U] \otimes W)^T=0$ for any finite dimensional rational $G$-module $W$.
This implies that $\Q[U]$   is an injective $G$-module (see \cite{Hoi})
We have $A^{\ast}_{a}(U)\cong {\rm Hom}_{\Q}(\bigwedge {\frak u}, \Q[U])$  and hence $A^{\ast}_{p}(U)\otimes V$ is an injective $G$-module (see \cite{Jan}).
Since the exponential map $\exp: {\frak u}\to U$ is an isomorphism of $\Q$-algebraic variety,  we have $H^{0}(A^{\ast}_{a}(U)\otimes V)=V$ and $H^{\ast}(A^{\ast}_{a}(U)\otimes V)=0$ for $\ast>0$.
Hence the sequence 
 \[\xymatrix{
V\ar[r]& A_{a}^{0}(U)\otimes V\ar^{d}[r]&A_{a}^{1}(U)\otimes V\ar^{d}[r]&\dots
 }
\] 
is an injective resolution.

By a splitting $G=T\ltimes U$, we have the homomorphism $\alpha : G\to {\rm Aut}(U)\ltimes U$.
We define the map 
\[\sigma^{p}(\cdot)(\cdot):\underbrace{G\times \dots\times G }_{p+1}\times \Q^{p}\to U\]
 such that
$\sigma ^{0}(g_{0})(0)=\alpha(g_{0})e$, $\sigma^{1}(g_{0},g_{1})(t_{1})=\alpha(g_{0}){\rm exp}((1-t_{1}){\rm log} \alpha(g_{0}^{-1}g_{1})e$ and inductively 
\[\sigma^{p}(g_{0},\dots, g_{p})(t_{1},\dots, t_{p})=\alpha(g_{0}){\rm exp}((1-t_{1}){\rm log} \sigma^{p-1}(g_{0}^{-1}g_{1},\dots, g_{0}^{-1}g_{p})(t_{2},\dots, t_{p})).
\]
It is known that the exponential map ${\rm exp}:\frak u\to U$ is an isomorphism of $\Q$-algebraic variety and ${\rm log} :U\to \frak u$ is the inverse.
Hence  the map 
\[\sigma^{p}:\underbrace{G\times \dots\times G }_{p+1}\times \Q^{p}\to U\]
 is a  homomorphism of $\Q$-algebraic variety such that 
for any $(t_{1},\dots,t_{p})$, the map 
\[\sigma^{p}(\cdot)(t_{1},\dots,t_{p}):\underbrace{G\times \dots\times G }_{p+1} \to U\]
 is a $G$-equivariant map.
We note 
\begin{multline*}
\sigma^{p}(g_{0},\dots,g_{p})(t_{1},\dots t_{i-1},0,t_{i},\dots t_{p-1})\\
=\sigma^{p-1}(g_{0},\dots ,g_{i-1},g_{i+1},\dots ,g_{p})(t_{1},\dots, t_{p-1}).
\end{multline*}

 For  $(g_{0},\dots,g_{p})$  and $\omega\in A^{\ast}_{a}(U)\otimes V$, considering the map $\sigma^{p}(g_{0},\dots,g_{p})(\cdot):\Q^{p} \to U\cong {\frak u}$ which is a homomorphism of  $\Q$-algebraic variety, we have the $\Q$-algebraic differential form $\sigma^{p}(g_{0},\dots,g_{p})^{\ast}\omega$ for parameters $(t_{1},\dots,t_{p})\in \Q^{p}$.
We regard $\sigma^{p}(g_{0},\dots,g_{p})^{\ast}\omega$ as a $\Q$-polynomial differential form on $\R^{p}$.
We define the map $\theta : A^{p}_{a}(U)\otimes V\to C^{p}(G,V)$ such that 
\[\theta (\omega)(g_{0},\dots,g_{p})=\int_{\Delta}\otimes {\rm id}_{V}\sigma^{p}(g_{0},\dots,g_{p})^{\ast}\omega\]
 where 
\[\Delta=\left\{(1-t_{1}-\dots-t_{p},t_{1},\dots,t_{p})\vert 0\le t_{i}\le 1  \right\}
\] is the standard $p$-simplex for the parameters $(t_{1},\dots,t_{p})$.
Then we can easily show that 
the map $\theta :   A^{\ast}_{a}(U)\otimes V\to C^{\ast}(G, V)$ is $G$-equivariant cochain complex homomorphism
by  Stokes' theorem. (cf.  \cite[Section 3]{Shul})

Now we have $(A^{\ast}_{a}(U)\otimes V)^{G}=\left(\bigwedge {\frak u}^{\ast}\otimes V\right)^{T}$
 where $\bigwedge {\frak u}^{\ast}\otimes V$ is the cochain complex of the Lie algebra ${\frak u}$ with values in the $\frak u$-module $V$.
Consider the restriction $\theta: \left(\bigwedge {\frak u}^{\ast}\otimes V\right)^{T}\to C^{\ast}(G, V)^{G}$.
Then the induced map $\theta:H^{\ast}({\frak u}, V)^{T}\to H^{\ast}(G,V)$ is identified with the map  ${\rm Ext}^{\ast}_{\mathcal G}(\Q,V)\to {\rm Ext}^{\ast}_{\mathcal G}(\Q,V)$ induced by the identity map $V\to V$.
Hence we have the following result.
\begin{theorem}\label{ALGCO}
The map $\theta: \left(\bigwedge {\frak u}^{\ast}\otimes V\right)^{T}\to C^{\ast}(G, V)^{G}$ induces a cohomology isomorphism
\[H^{\ast}({\frak u}, V)^{T}\cong H^{\ast}(G,V).
\]
\end{theorem}

\section{Simplicial  de Rham theory}\label{des}
We denote by $A^{\ast}_{p}(n)$ the $\Q$-DGA which is generated by $t_{0},\dots, t_{n}$ of degree $0$ and $dt_{0},\dots, dt_{n}$ of degree $1$ with the relations $t_{0}+\dots+t_{n}=1$ and $dt_{0}+\dots+dt_{n}=0$.
We can regard  $A^{\ast}_{p}(n)$ as the  $\Q$-polynomial  de Rham complex on the standard $n$-simplex $\Delta^{n}$.
We define the map $\int_{\Delta^{n}} : A^{n}_{p}(n)\to \Q$ by the ordinary Riemannian integral.
Let $K$ be a simplicial complex with a universal covering complex $\tilde K$.
We denote by $n(\sigma)$ the dimension of a simplex $\sigma\in K$.
Let $V$ be a finite-dimensional $\Q$-vector space which is a $\pi_{1}(K)$-module.
We denote by  $A^{\ast}_{p}( K, V)$ the space of collections
$\{\omega_{\sigma}\in A^{\ast}_{p}(n(\sigma))\otimes V \}_{\sigma\in \tilde K}$ such that:
\begin{itemize}
\item $\{\omega_{\sigma}\}_{\sigma\in \tilde K}$ are  compatible under restrictions  to faces i.e. $i^{\ast}\omega_{\sigma}=\omega_{\tau}$ for the inclusion $i:\tau\to \sigma$ of a face.
\item $\{\omega_{\sigma}\}_{\sigma\in \tilde K}$ are invariant under the $\pi_{1}(K)$-action i.e.
$\gamma\cdot \omega_{\gamma\sigma}=\omega_{\sigma}$ for any $\gamma\in \pi_{1}(K)$.
\end{itemize}
We call  $A^{\ast}_{p}(K, V)$ the $\Q$-polynomial de Rham complex of $K$ with values in local system $V$.
The space $A^{\ast}_{p}(K, V)$ with the exterior derivation is a cochain complex.
Let $C^{\ast}(K,V)=(C^{\ast}(\tilde K)\otimes V)^{\Gamma}$ be the cochain complex of simplicial cochains with values in the local system $V$.
Define the map $\iota: A^{n}_{p}( K, V)\to C^{n}(K,V)$ such that for $\sigma\in K$ with $n(\sigma)=n$
\[\iota(\{\omega\})(\sigma)=\int_{\Delta^{n}}\otimes {\rm id}_{V} (\omega_{\sigma}).
\]
Then, 
this map is a cochain complex homomorphism and this map induces a cohomology isomorphism
(see \cite[Chapter 9]{GMo},  \cite[Chapter 12-14]{Halp}).

Let $V=\varinjlim V_{i}$ for an inductive system of finite dimensional $\pi_{1}(K)$-modules.
We define
\[A^{\ast}_{p}(K, V)=\varinjlim A^{\ast}_{p}(K, V_{i})
\]

\begin{example}\label{HainDG}{\rm(cf.\cite{Hai}) }
Let $T$ be a reductive $\Q$-algebraic group and $\rho:\pi_{1}(K)\to T$ a representation.
Consider the coordinate ring $\Q[T]$   of $T$.
Then as a $(T,T)$-bimodule,
we have 
\[\Q[T]=\bigoplus V_{\alpha}^{\ast}\otimes V_{\alpha}\]
such that $\{V_{\alpha}\}$ is a set of isomorphism classes of irreducible right $T$-module (\cite[Proposition 3.1]{Hai}).
We regard $\Q[T]$ as a $\pi_{1}(K)$-module by $\rho$.
Then we have 
\[A^{\ast}_{p}(K, \Q[T])=\bigoplus A^{\ast}_{p}(K, V_{\alpha}^{\ast})\otimes V_{\alpha}
\] 
and it is a DGA.

\end{example}


Let $\Gamma$ be a discrete group.
For a $\Gamma$-module $V$, we define the group cohomology $H^{\ast}(\Gamma,V)={\rm Ext}^{\ast}_{\Gamma}(\Q,V)$.
We have the standard resolution.
We denote by $C^{p}(\Gamma,V)$ the set of the $V$-valued functions on $\underbrace{\Gamma\times \dots\times \Gamma }_{p+1}$ with the left-$\Gamma$-action.
Consider the sequence
\[\xymatrix{
V\ar[r]& C^{0}(\Gamma,V)\ar^{d}[r]&C^{1}(\Gamma,V)\ar^{d}[r]&\dots
 }
\] 
such that the first map $V\to C^{0}(\Gamma,V)$ is the embedding as  constant functions and $d:C^{p}(\Gamma,V)\to C^{p+1}(\Gamma,V)$ is given by
\[d\phi (\gamma_{0},\dots,\gamma_{p+1})=\sum (-1)^{i}\phi(\gamma_{0},\dots, {\hat \gamma}_{i},\dots, \gamma_{p+1})
\]
for $\phi\in C^{p}(\Gamma,V)$, $\gamma_{0},\dots, \gamma_{p+1}\in\Gamma$.
Then the group cohomology $H^{\ast}(\Gamma,V)$ is the cohomology of the cochain complex $C^{\ast}(\Gamma,V)^{\Gamma}$.
A representation $\rho:\Gamma\to G$ induces a homomorphism $\rho^{\ast}:H^{\ast}(G,V)\to H^{\ast}(\Gamma,V)$.
By the map $\theta: \left(\bigwedge {\frak u}^{\ast}\otimes V\right)^{T}\to C^{\ast}(G, V)^{G}$ as Section \ref{SAL},
we give a geometric representation of $\rho^{\ast}:H^{\ast}(G,V)\to H^{\ast}(\Gamma,V)$.

We define the acyclic simplicial complex $E\Gamma$ with the free discontinuous $\Gamma$-action so that:
\begin{itemize}
\item For integers $n\ge 0$, simplices of $E\Gamma$ are standard $n$-simplices $\Delta_{(\gamma_{0},\dots,\gamma_{n})}$ indexed by $\Gamma^{n+1}$.
\item \[(t_{1},\dots,t_{i-1},t_{i+1},\dots,t_{n})_{(\gamma_{0},\dots,\gamma_{i-1},\gamma_{i+1},\dots,\gamma_{n})}\in \Delta_{(\gamma_{0},\dots,\gamma_{i-1},\gamma_{i+1},\dots,\gamma_{n})}\] is identified with
 \[(t_{1},\dots,t_{i-1},0,t_{i+1},\dots,t_{n})_{(\gamma_{0},\dots,\gamma_{n})}\in \Delta_{(\gamma_{0},\dots,\gamma_{n})}.\]
\item For $\gamma\in\Gamma$, the action is given by
\[\gamma\cdot (t_{1},\dots,t_{n})_{(\gamma_{0},\dots,\gamma_{n})}=(t_{1},\dots,t_{n})_{(\gamma\gamma_{0},\dots,\gamma\gamma_{n})}.
\]

\end{itemize}
We define  $B\Gamma$ the quotient of $E\Gamma$ by the $\Gamma$-action.
Then the simplicial complex $B\Gamma$ is an Eilenberg-Maclane space $K(\Gamma,1)$.
For a finite dimensional $\Gamma$-module $V$, we consider  the $\Q$-polynomial de Rham complex $A^{\ast}_{p}(B\Gamma,V)$ of $B\Gamma$ with values in local system $V$.
Define the map $\iota:  A^{n}_{p}(B\Gamma,V)\to C^{n}(\Gamma,V)^{\Gamma}$ as 
\[\iota(\{\omega_{\sigma}\}_{\sigma\in B\Gamma} )(\gamma_{0},\dots,\gamma_{n})=\int_{\Delta_{(\gamma_{0},\dots,\gamma_{n})}}\omega_{\Delta_{(\gamma_{0},\dots,\gamma_{n})}}.
\]
Since we can identify $C^{\ast}(\Gamma,V)^{\Gamma}$ with the cochain complex  $C^{\ast}(B\Gamma,V)$, the map $\iota:  A^{p}_{p}(B\Gamma,V)\to C^{p}(\Gamma,V)^{\Gamma}$ induces a cohomology isomorphism
\[H^{\ast}(A^{\ast}_{p}(B\Gamma,V))\cong H^{\ast}(\Gamma,V).
\]
We define the map $\psi:\left(\bigwedge {\frak u}^{\ast}\otimes V\right)^{T}\to  A^{\ast}_{p}(B\Gamma,V)$ such that 
\[\psi(\omega)=\{\sigma^{p}(\rho(\gamma_{0}),\dots,\rho(\gamma_{p}))^{\ast}\omega\}_{\Delta_{(\gamma_{0},\dots,\gamma_{n})}}\]
where $\sigma^{p}(\rho(\gamma_{0}),\dots,\rho(\gamma_{p}))^{\ast}\omega$ is defined in Section \ref{SAL}.
By the $G$-invariance of $\omega\in \left(\bigwedge {\frak u}^{\ast}\otimes V\right)^{T}$   and the relation
\begin{multline*}\sigma^{p}(g_{0},\dots,g_{p})(t_{1},\dots t_{i-1},0,t_{i},\dots t_{p-1})
\\
=\sigma^{p-1}(g_{0},\dots ,g_{i-1},g_{i+1},\dots ,g_{p})(t_{1},\dots, t_{p-1}),
\end{multline*}
we actually have $\psi(\omega)\in  A^{\ast}_{p}(B\Gamma,V)$.
Then 
the map $\psi:\left(\bigwedge {\frak u}^{\ast}\otimes V\right)^{T}\to  A^{\ast}_{p}(B\Gamma,V)$ is a cochain complex homomorphism.

We have the commutative diagram
\[\xymatrix{
\left(\bigwedge {\frak u}^{\ast}\otimes V\right)^{T}\ar[d]^{\psi}\ar[r]^{\theta}&C^{\ast}( G,V)^{G}\ar[d]^{\rho^{\ast}}\\
A^{\ast}_{p}(B\Gamma,V)\ar[r]^{\iota}& C^{p}(\Gamma,V)^{\Gamma}.
}
\]
Hence we have:
\begin{corollary}\label{cocBB}
The induced map $\psi^{\ast}:H^{\ast}({\frak u}, V)^{T}\to  H^{\ast}(A^{\ast}_{p}(B\Gamma,V))$ is identified with the map $\rho^{\ast}:H^{\ast}(G,V)\to H^{\ast}(\Gamma,V)$.
\end{corollary}
Suppose $V=\Q[T]$.
Then we have $\left(\bigwedge {\frak u}^{\ast}\otimes \Q[T]\right)^{T}=\bigwedge {\frak u}^{\ast}$ and we have the DGA map $\psi:\bigwedge {\frak u}^{\ast}\to A^{\ast}_{p}(B\Gamma,\Q[T])$.
As we notice in Example \ref{HainDG},
in this case Corollary \ref{cocBB} is written as the following statement.
\begin{corollary}\label{Hni}
The  induced map $\psi^{\ast}:H^{\ast}({\frak u}, \Q)\to  H^{\ast}(A^{\ast}_{p}(B\Gamma,\Q[T]))$ is identified with the map \[\rho^{\ast}:\bigoplus H^{\ast}(G,V_{\alpha}^{\ast})\otimes V_{\alpha}\to \bigoplus H^{\ast}(\Gamma,V_{\alpha}^{\ast})\otimes V_{\alpha}.\]

\end{corollary}

\section{Classifying spaces of torsion-free virtually polycyclic groups}\label{poll}
Let $\Gamma$ be a torsion-free virtually polycyclic group,
  $G$ a $\Q$-algebraic group and  $\rho:\Gamma\to  G$ a representation with the Zariski-dense image.
It is known that we have $\dim U\le {\rm rank}\,\Gamma$ where $U$ is the  unipotent radical of $G$.
We say that  $\rho:\Gamma\to  G$ is a full representation if $\dim U= {\rm rank}\,\Gamma$.
\begin{theorem}[\cite{KCT}]\label{kasct}
If $\rho:\Gamma\to  G$ is an injective full representation, then for any rational $G$-module $V$ the induced map 
 $\rho^{\ast}:H^{\ast}(G,V)\to H^{\ast}(\Gamma,V)$
is an isomorphism.
\end{theorem}
We suppose that $\rho:\Gamma\to  G$ is an injective full representation.
Denote by $\frak u$  the Lie algebra of  the unipotent radical  $U$ of $G$.
Take a splitting $G=T\ltimes U$ for a maximal reductive subgroup $T$.
Then by Corollary \ref{cocBB} and Theorem \ref{kasct}, we have the following fact.

\begin{theorem}\label{indisoinn}
 The map $\psi:\left(\bigwedge {\frak u}^{\ast}\otimes V\right)^{T}\to  A^{\ast}_{p}(B\Gamma,V)$ induces a cohomology isomorphism.
\end{theorem}
Suppose  $V=\Q[T]$ as a $\Gamma$-module.
By Corollary \ref{Hni}, we obtain the following result.

\begin{theorem}\label{polmi}
We have an explicit DGA map $\bigwedge {\frak u}^{\ast}\to  A^{\ast}_{p}(B\Gamma,\Q[T])$ which induces a cohomology isomorphism.
Hence $\bigwedge{\frak u}^{\ast}$ is the minimal model of $A^{\ast}_{p}(B\Gamma,\Q[T])$.
\end{theorem}

\section{Examples and applications to symplectic and contact geometry}\label{Exss}

 An infra-solvmanifold is a manifold of the form $G/\Delta$, where $G$ is a simply connected solvable Lie group, and $\Delta$ is a torsion-free subgroup of ${\rm Aut}(G)\ltimes G$ such that the closure of $h(\Delta)$ in ${\rm Aut}(G)$ is compact where $h:{\rm Aut}(G)\ltimes G\to {\rm Aut}(G)$ is the projection.
An infra-solvmanifold is a  generalization of a solvmanifold.
An infra-solvmanifold is a aspherical manifold with a virtually polycyclic group.

Let $\Gamma$ be a torsion-free virtually polycyclic group.
Then there exists a compact infra-solvmanifold $M$ with the fundamental group $\Gamma$ (see \cite{B}).
It is known that every compact infra-solvmanifold is smoothly rigid   (\cite[Corollary 1.5]{B}).
Hence we have the canonical correspondence between   torsion-free virtually polycyclic groups $\Gamma$ and  infra-solvmanifolds $M_{\Gamma}$ with $\pi_{1}(M_{\Gamma})=\Gamma$. 
Hence by Theorem \ref{indisoinn}, we can study the cohomology of a infra-solvmanifolds $M_{\Gamma}$.

For a  torsion-free virtually polycyclic group $\Gamma$ and a full representation $\rho:\Gamma\to G$,
take  a splitting $G=T\ltimes U_{\Gamma}$ for a maximal reductive subgroup $T$.
Denote by $\frak u$  the Lie algebra of  the unipotent radical $U$ of $G$.
Then by Theorem \ref{indisoinn}, the cohomology of the DGA 
$\left(\bigwedge {\frak u}^{\ast}\right)^{T}$ is isomorphic to $H^{\ast}(M_{\Gamma},\Q)$.
By this isomorphism, we can find an integral symplectic form on the compact infra-solvmanifold $M_{\Gamma}$.

\begin{proposition}
If there exists a $2$-form $\omega\in \left(\bigwedge^{2} {\frak u}^{\ast}\right)^{T}$ such that $d\omega=0$ and $\omega$ is non-degenerate,
 then the   compact infra-solvmanifold $M_{\Gamma}$ admits a integral symplectic form.
\end{proposition}
\begin{proof}
We suppose that ${\rm rank}\, \Gamma=2n$.
It is sufficient to show that $M_{\Gamma}$ admits a symplectic form $\alpha$ such that $[\alpha]\in H^{2}(M_{\Gamma},\Q)$.
It is known that if we have $a\in H^{2}(M_{\Gamma},\R)$ so that $a^{n}\not=0$, then we have a symplectic form $\alpha$ which is a representative of $a$ (see \cite{Ksy}).

Now we suppose that
 there exists a $2$-form $\omega\in \left(\bigwedge^{2} {\frak u}^{\ast}\right)^{T}$ such that $d\omega=0$ and $\omega$ is non-degenerate.
Consider the cohomology class $[\omega]\in H^{2}(M_{\Gamma},\Q)$.
Then, since $\omega\in \left(\bigwedge^{2} {\frak u}^{\ast}\right)^{T}$ is non-degenerate, we have $[\omega]^{n}\not=0$.
Hence, taking a symplectic form which is a representative of $[\omega]$, we can prove the proposition.
\end{proof}

\begin{example}
Let $A$ be a semi-simple matrix such that $A\in SL_{m}(\Z)$ and for any non-zero integer $n$ we suppose $A^{n}\not=1$ where $1$ is the unit matrix.
Consider the semi-direct product $\Gamma=\Z\ltimes \Z^{m}$ such that 
the action of $\Z$ on $\Z^{m}$ is given by 
$\Z\ni t\mapsto A^{t}\in {\rm Aut}(\Z^{m})$.
The group $\Gamma$ is torsion-free polycyclic of rank $m+1$.
Take $B=\left(
\begin{array}{ccccccc}
1&0\\
0&A
\end{array}
\right) \in  SL_{m+1}(\Z)$.
Let $T$ be the Zariski-closure of $\langle B\rangle $ in $SL_{m+1}(\Q)$.
By the assumption, $T$ is diagonalizable.
consider the $\Q$-algebraic group $G=T\ltimes \Q^{m+1}$ with the unipotent radical $\Q^{m+1}$.
We have the injective homomorphism
\[\rho: \Gamma \ni (t,v)\mapsto (B^{t},t,v)\in G
\]
where $t\in \Z$ and $v\in \Z^{m}$.
Then we can easily check that $\rho(\Gamma)$ is Zariski-dense in $G$ and so $\rho$ is full.  
Thus by  Theorem \ref{indisoinn}, we have an isomorphism
\[H^{\ast}(M_{\Gamma},\Q)\cong \left(\bigwedge \Q^{m+1}\right)^{T}=\left(\bigwedge \Q^{m+1}\right)^{\langle B\rangle}.
\]

Additionally, we assume that 
\[A=\left(
\begin{array}{ccccccc}
1&&&\\
&A_{1}&&\\
&&\ddots&\\
&&&A_{k}
\end{array}\right)
\]
such that $A_{1},\dots, A_{k}\in SL_{2}(\Z)$ and they are semi-simple.
By the above assumption, some $A_{i}$ satisfies $A^{n}_{i}\not=1$ for any non-zero integer $n$.
Then, taking the standard basis $e_{1}\dots e_{2k+2}$ of $\Q^{2k+2}$,
we obtain the non-degenerate two form $\omega\in  \left(\bigwedge^{2} \Q^{2k+2}\right)^{\langle B\rangle}$
such that 
\[\omega=e_{1}\wedge e_{2}+\dots +e_{2k+1}\wedge e_{2k+2}.
\]
Hence in this case $M_{\Gamma}$ admits a integral  symplectic form $\omega$.
By  \cite[Proposition 1.1, Proposition 1.4]{Kfo}, we can say that $M_{\Gamma}$ is formal in the sense of Sullivan (\cite{DGMS}, \cite{Sul}) and $(M_{\Gamma},\omega)$ satisfies the hard Lefschetz property.
On the other hand, since the matrix $A$ does not have a finite period, $\Gamma$ is not virtually nilpotent and hence by the result in \cite{AN} $M_{\Gamma}$ does not admit a K\"ahler structure.
Using integral symplectic form $\omega$, we construct the following examples:
\begin{newexample}

Since $M_{\Gamma}$ is $2k+2$-dimensional manifold with integral symplectic form $\omega$,
we have a symplectic emmbedding $M_{\Gamma}\to \C P^{N}$ for $N\ge 2k+3$ (\cite{Gro}, \cite{Ts}).
Hence, as McDuff's construction \cite{Mcd}, we obtain the symplectic blow-up $X_{\Gamma}$ of $\C P^{N}$ along $M_{\Gamma}$.
By \cite[Theorem 2.2]{Cav}, the symplectic blow-up $X_{\Gamma}$ satisfies the hard Lefschetz property.
Moreover, in \cite{LS}, it is announced that the blow-up along a manifold $M$ symplectically embedded in a large enough complex projective space is formal if and only if $M$ is formal.
If this fact true, then $X_{\Gamma}$ is formal and hence we obtain a formal and hard Lefschetz symplectic manifolds such that we do not know whether they admit K\"ahler structures.
\end{newexample}
\begin{newexample}
By the integral symplectic form $\omega$ on $M_{\Gamma}$, 
we obtain the principal circle bundle $P_{\Gamma}\to M_{\Gamma}$ associated with $[\omega]\in H^{2}(M_{\Gamma},\Z)$.
It is known that $P_{\Gamma}$ is a regular contact manifold.
Recently, the hard Lefschetz property on contact manifolds is defined. (see \cite{CNY}, \cite{Li})
In \cite{Li}, it is proved that the contact manifold $P_{\Gamma}$ satisfies hard Lefschetz property if and only if the symplectic manifold $M_{\Gamma}$ satisfies hard Lefschetz property.
Hence $P_{\Gamma}$ satisfies hard Lefschetz property.
In \cite{KSS}, it is proved that polycyclic fundamental groups of  compact Sasakian manifolds are virtually nilpotent.
Since $\Gamma$ is not virtually nilpotent, the fundamental group of $P_{\Gamma}$ is  not virtually nilpotent.
Hence $P_{\Gamma}$ does not admit a Sasakian structure.

\end{newexample}
\end{example}

\end{document}